\documentclass{article}

\usepackage[top=30truemm,bottom=30truemm,left=25truemm,right=25truemm]{geometry}

\usepackage[nobysame,alphabetic]{amsrefs}

\usepackage{amsmath,amsthm,amssymb,graphicx,enumitem}
\usepackage{tikz}
\usetikzlibrary{arrows,decorations.pathmorphing,backgrounds,positioning,fit,petri,patterns}

\usepackage[all]{xy}

\usepackage[vcentermath]{youngtab}

\theoremstyle{definition}
\newtheorem{defi}{Definition}[section]

\theoremstyle{plain}
\newtheorem{lemm}[defi]{Lemma}

\newtheorem{prop}[defi]{Proposition}
\newtheorem{theo}[defi]{Theorem}
\newtheorem{coro}[defi]{Corollary}

\theoremstyle{remark}
\newtheorem{rema}[defi]{Remark}
\newtheorem*{rema*}{Remark}
\newtheorem{exam}[defi]{Example}
\newtheorem*{exam*}{Example}

\newcommand{\PP}{\mathcal{P}}

\newcommand{\la}{\lambda}
\newcommand{\La}{\Lambda}
\newcommand{\ka}{\kappa}

\newcommand{\de}{\delta}
\newcommand{\DE}[1]{\de\left[#1\right]}

\newcommand{\Z}{\mathbb{Z}}

\newcommand{\lra}{\longrightarrow}

\newcommand{\sm}{\setminus}

\newcommand{\ti}{\tilde}

\renewcommand{\emptyset}{\varnothing}

\newcommand{\tikzitem}[2][1.0]{
	\begin{tikzpicture}[scale=#1]
		#2
	\end{tikzpicture}
}

\newcommand{\tikzbox}[3][]{
	\draw [#1] (#3,#2) rectangle +(-1,-1);
}

\newcommand{\tikzydiag}[2][]{
	\foreach \li [count=\i] in {#2}{
		\foreach \j in {1,...,\li}{
			\tikzbox[#1]{\i}{\j};
		}
	}
}
\newcommand{\tikzyfill}[2][gray]{
	\foreach \li [count=\i] in {#2}{
		\foreach \j in {1,...,\li}{
			\fill[#1] (\j,\i) rectangle +(-1,-1);
		}
	}
}

\usepackage{todonotes}
\newcommand{\Todo}[1]{}

\usepackage{hyperref}

\usepackage{stmaryrd}
\usepackage{ifthen}

\numberwithin{equation}{section}

\newcommand{\wh}{\widehat}

\newcommand{\Gr}{\mathrm{Gr}}
\newcommand{\RPP}{\mathrm{RPP}}
\newcommand{\id}{\mathrm{id}}
\newcommand{\Hom}{\mathrm{Hom}}

\newcommand{\Kth}[1][k,n]{K^*(\Gr(#1))}
\newcommand{\Kho}[1][k,n]{K_*(\Gr(#1))}

\newcommand{\mc}{\mathcal}
\newcommand{\mcO}{\mc{O}}
\newcommand{\mcI}{\mc{I}}

\newcommand{\skewyd}[3][1.0]{
	\tikzitem[#1]{
		\tikzyfill{#3}
		\tikzydiag{#2}
	}
}

\newcommand{\dgYoungDiagSize}{.2}
\newcommand{\yd}[1]{
	\ifthenelse{\equal{#1}{}}{
		\emptyset
	}
	{
		\skewyd[\dgYoungDiagSize]{#1}{}
	}
}
\newcommand{\sk}[2]{\skewyd[\dgYoungDiagSize]{#1}{#2}}
\newcommand{\dg}[1]{g_{\yd{#1}}}
\newcommand{\dgs}[2]{g_{\sk{#1}{#2}}}

\newcommand{\Itemize}[1]{
	\begin{itemize}
		\setlength{\parskip}{0cm} 
		\setlength{\itemsep}{0cm} 
		#1
	\end{itemize}
}
\newcommand{\Enumerate}[1]{
	\begin{enumerate}[label=\textup{(\arabic*)}]
		\setlength{\parskip}{0cm} 
		\setlength{\itemsep}{0cm} 
		#1
	\end{enumerate}
}

\def\para{%
	\setlength{\unitlength}{1pt}%
	\thinlines %
	\begin{picture}(12, 12)%
	\put(0,0){/}
	\put(2,0){/}
	\end{picture}%
	\hspace{-1.5mm}%
}%

\newcommand{\TitleContent}{Automorphisms on the ring of symmetric functions and stable and dual stable Grothendieck polynomials}

\title{\TitleContent}
\author{Motoki Takigiku}
\date{\today}

\AtEndDocument{\bigskip{\footnotesize%
  \textsc{Graduate School of Mathematical Sciences, the University of Tokyo, Japan} \par  
  \textit{E-mail address}: \texttt{takigiku@ms.u-tokyo.ac.jp} 
}}

\newcommand{\AcknowledgementContent}{
	The author is grateful to
	Itaru Terada
	for many valuable discussions and comments.
	This work was supported by
	the Program for Leading Graduate
	Schools, MEXT, Japan.
}

\newcommand{\AbstractContent}{
	The dual stable Grothendieck polynomials $g_\lambda$ and their sums $\sum_{\mu\subset\lambda} g_\mu$ (which represent $K$-homology classes of boundary ideal sheaves and structure sheaves of Schubert varieties in the Grassmannians) have the same product structure constants. In this paper we first explain that the ring automorphism $g_\lambda\mapsto\sum_{\mu\subset\lambda} g_\mu$ on the ring of symmetric functions is described as the operator $F^\perp$, the adjoint of the multiplication $(F\cdot)$, by a ``group-like'' element $F=\sum_{i} h_i$ where $h_i$ is the complete symmetric function. Next we give a generalization: starting with another ``group-like'' elements $\sum_{i} t^i h_i$, we obtain a deformation with a parameter $t$ of the ring automorphism above, as well as identities involving stable and dual stable Grothendieck polynomials.%
}

\begin{document}

\maketitle

\begin{abstract}
	\AbstractContent
\end{abstract}


\section{Introduction}

The dual stable Grothendieck polynomials $g_\la$
are a certain family of inhomogeneous symmetric functions
parametrized by interger partitions $\la$.
These functions are a $K$-theoretic deformation of the Schur functions,
and dual to another deformation called
stable Grothendieck polynomials $G_\la$.

Historically
the stable Grothendieck polynomials (parametrized by permutations)
were introduced by Fomin and Kirillov \cite{MR1394950}
as a stable limit of the Grothendieck polynomials of
Lascoux--Sch\"utzenberger \cite{MR686357}.
In \cite{MR1946917} Buch gave a 
combinatorial formula for the stable Grothendieck polynomials $G_\la$
for partitions using so-called set-valued tableaux,
and showed that their span
$\bigoplus_\la \Z G_\la$ is a bialgebra
and that
the $K$-theory
of a Grassmannian
is isomorphic to a quotient of it.

The dual stable Grothendieck polynomials $g_\la$
were
introduced by Lam and Pylyavskyy \cite{MR2377012}
as generating functions of reverse plane partitions
(Definition \ref{defi:gla}),
and shown to be the dual basis for 
$G_\la$
via the Hall inner product
and represent the $K$-homology classes of the ideal sheaves of boundaries of Schubert varieties in Grassmannians.

As seen in Remark \ref{rema:Gg_geom},
the sums $\sum_{\mu\subset\la}g_\mu$ represent 
the classes in $K$-homology of structure sheaves of Schubert varieties.
%
In \cite{dualstablesum} the author showed that
the bases $\{g_\la\}_\la$ and $\{\sum_{\mu\subset\la}g_\mu\}_\la$ have the same product structure constants,
i.e.\,the linear map $g_\la\mapsto\sum_{\mu\subset\la}g_\mu$, which we denote by $I$, is a ring automorphism on
the ring of symmetric functions $\La$,
by showing the Pieri rules with respect to these bases have the same coefficients.

In this paper we first give alternative descriptions for this map $I$:
it is realized as the substitution $f(x_1,x_2,\cdots)\mapsto f(1,x_1,x_2,\cdots)$ \eqref{eq:I=i*id}.
The key observation used there is
that 
the substitution $f\mapsto f(1,0,0,\cdots)$
maps
$g_{\la/\mu}$ to $1$
for any skew shape $\la/\mu$ (Proposition \ref{theo:i});
then $I$ is a certain composition of this map and the coproduct on $\La$.
We also give two formulas for the image of $g_{\la/\mu}$ under $I$ \eqref{eq:I_skew},
which generalizes $I(g_\la)=\sum_{\nu\subset\la}g_\nu$ and
implies some identities involving skew dual stable Grothendieck polynomials.

Next we give another description using the Hopf structure on $\La$ more explicitly.
Due to the coalgebra structure of $\La$ 
and the identification $\La^*\simeq\wh\La$ (completion of $\La$) via the Hall inner product,
for any $F\in\wh\La$ there corresponds a linear map $F^\perp\colon\La\lra\La$,
the adjoint of the multiplication map $(F\cdot)\colon\wh\La\lra\wh\La$.
We explain in Lemma \ref{theo:group-like} that 
$F^\perp$ is a ring automorphism 
if and only if
$F$ has ``group-like'' property.
Then we show that the map $I$ is realized as $H(1)^\perp$ 
where $H(1):=\sum_{i}h_i=\sum_{\la} G_\la$ is group-like (Theorem \ref{theo:H(1)_overall}),
while $I^{-1}=E(-1)^\perp$ where $E(-1):=\sum_{i}(-1)^i e_i=1- G_1$.
Besides we explain that
from presentations of the map $(H(1)\cdot)$
with respect to the basis $\{G_{\la}\}$
we can obtain identities involving $G_\la$, \eqref{eq:I*(1)} and \eqref{eq:D*(1)}.

Finally we give a generalization:
starting with another group-like elements $H(t)=\sum_{i\ge 0} t^i h_i$ and $E(t)=\sum_{i\ge 0} t^i e_i$,
we obtain a deformation with a parameter $t$ of the automorphism $I$
and hence the identities obtained above involving $G_\la$ and $g_\la$.

\subsection*{Organization}
In Section \ref{sect:Prel} we recall 
Hopf algebras, symmetric functions 
and stable Grothendieck and dual stable Grothendieck polynomials.
In Section \ref{sect:algebraic}
we see connection between group-like elements in $\wh\La$ and 
automorphisms on $\La$,
as well as our main examples of group-like elements $H(t)$ and $E(t)$.
Section \ref{sect:I} 
contains the first main results:
after proving $I(f(x))=f(1,x)$ in Section \ref{sect:I:f(1,x)},
we show $I=H(1)^\perp$ and $I^{-1}=E(-1)^\perp$ in Section \ref{sect:I:H1_perp}.

In following two sections we give a generalization.
In Section \ref{sect:Ht}
we give descriptions for 
the automorphism $H(t)^\perp$ and the multiplication map $(H(t)\cdot)$
using the bases $\{g_{\la}\}$ and $\{G_\la\}$, respectively,
while
in Section \ref{sect:Et}
we treat $E(t)^\perp$ and $(E(t)\cdot)$.
In Section \ref{sect:eg}
we give some examples.

Note that,
although 
the results in Section \ref{sect:I} are special cases of
that of Section \ref{sect:Ht} and \ref{sect:Et},
they need to be shown first
according to the proof given in this paper.

\Todo{check: $\mathrm{char}(K)=?$}
\Todo{involution $\omega$}

\subsection*{Acknowledgement}

\AcknowledgementContent

\section{Preliminaries}\label{sect:Prel}

Throughout this paper,
we fix a commutative ring $K$ and 
assume that any modules, algebras, morphisms etc.\ are over $K$.

\subsection{Hopf algebra}\label{sect:Prel::Hopf}

We recall some generalities on the Hopf algebra.
For more details
we refer the reader to \cite{MR0252485,MR594432,1409.8356} for example.

An {\em algebra} $A$ is a $K$-module
equipped with
a {\em product} (or {\em multiplication}) $m=m_A\colon A\otimes A\lra A$ and
a {\em unit} $u=u_A\colon K\lra A$
satisfying
$m\circ(m\otimes\id) = m\circ(\id\otimes m)$ and
$m\circ(\id\otimes u)=\id=m\circ(u\otimes\id)$.
A {\em coalgebra} $C$ is a $K$-module
equipped with
a {\em coproduct} (or {\em comultiplication}) $\Delta=\Delta_C\colon C\lra C\otimes C$ and
a {\em counit} $\epsilon=\epsilon_C\colon C\lra K$
satisfying
$(\Delta\otimes\id)\circ\Delta = (\id\otimes \Delta)\circ\Delta$ and
$(\id\otimes \epsilon)\circ\Delta=\id=(\epsilon\otimes\id)\circ\Delta$.
A $K$-linear map $\varphi\colon A\lra B$ between algebras is
an {\em algebra morphism} if
$\varphi\circ m_A = m_B\circ(\varphi\otimes\varphi)$ and
$\varphi\circ u_A = u_B$.
%
A $K$-linear map $\varphi\colon C\lra D$ between coalgebras is
a {\em coalgebra morphism} if
$(\varphi\otimes\varphi)\circ\Delta_C = \Delta_D\circ\varphi$ and
$\epsilon_C = \epsilon_D\circ\varphi$.
%
A $K$-module $A$ equipped with $m,u,\Delta,\epsilon$
is a {\em bialgebra} if
$(A,m,u)$ is an algebra, 
$(A,\Delta,\epsilon)$ is a coalgebra, and
the following equivalent conditions hold:
(a) $\Delta,\epsilon$ are algebra morphisms;
(b) $m, u$ are coalgebra morphisms.
A bialgebra $A$ equipped with an \emph{antipode} map
$S\colon A\lra A$ satisfying
$m\circ(S\otimes\id)\circ\Delta=u\circ\epsilon$
is called a \emph{Hopf algebra}.

\subsubsection{duals}

For a $K$-module $A$,
let $A^*=\Hom(A,K)=\{f\colon A\lra K\colon K\text{-linear}\}$ and
$(\,,)=(\,,)_A\colon A^*\times A\lra K\,;\,(f,a)=f(a)$.
For a graded $K$-module $A=\bigoplus_{n\ge 0}A_n$,
we denote by $A^o$ the \emph{graded dual} $\bigoplus_{n}A_n^*$,
and $A$ is called {\em of finite type} if every $A_n$ is a finite free $K$-module.
For any coalgebra $C$, its dual $C^*$ is an algebra by
\begin{equation}\label{eq:dual_copro}
(m_{C^*}(f\otimes g), a)_C = (f\otimes g, \Delta_C(a))_{C\otimes C}
\end{equation}
for $f,g\in C^*$ and $a\in C$.
If an algebra $A$ is a finite free $K$-module 
(resp.\,$A$ is a graded algebra of finite type),
then its dual $A^*$ (resp.\,graded dual $A^o$) is a coalgebra by
\begin{equation}\label{eq:dual_pro}
(\Delta_{A^*}(f), a\otimes b)_{A\otimes A}=(f,ab)_{A}
\end{equation}
for $f\in A^*$ (resp.\,$A^o$) and $a,b\in A$.
\Todo{
Assume an algebra $A$ is finite free or of finite type and
let $\{x_i\}_{i\in I}$ be a basis of $A$ and
$\{f_i\}_{i\in I}$ be the dual basis of the coalgebra $A^*$ (or $A^o$).
Then the product structure constants for $\{x_i\}$ coincide with
the coproduct structure constants for $\{f_i\}$:
if $x_i x_j=\sum_{k} a_{ij}^k x_k$ then
$\Delta(f_k)=\sum_{i,j} a_{ij}^k f_i\otimes f_j$.}

For a coalgebra $C$ and an algebra $A$, the space of linear maps
$\Hom(C,A)$ becomes an associative algebra by the \emph{convolution product} $*$
defined by 
$f*g = m_A\circ(f\otimes g)\circ\Delta_C$.
Then $u_A\circ\epsilon_C$ is the identity for $*$,
and the convolution product on $C^*=\Hom(C,K)$ coincides with
the product given in \eqref{eq:dual_copro}.

\subsubsection{Module and comodule morphisms}\label{sect:Prel::Hopf::comod}

For a coalgebra $C$,
a linear map
$\phi\colon C\lra C$ is \emph{$C$-comodule morphism} if 
$\Delta\circ\phi=(\phi\otimes\id)\circ\Delta$.
For an algebra $A$,
a linear map
$\psi\colon A\lra A$ is \emph{$A$-module morphism} if 
$\psi\circ m=m\circ(\psi\otimes\id)$.

\begin{lemm}\label{theo:comodule_hom}
	Let $C$ be a coalgebra and $C^*$ its dual algebra.
	For a linear map $\phi\colon C\lra C$, the following are equivalent:
	$(1)$ $\phi\colon C\lra C$ is a $C$-comodule morphism.
	$(2)$ $\phi^*\colon C^*\lra C^*$ is a $C^*$-module morphism.
\end{lemm}
\begin{proof}
	It easily follows from
	$(f,\phi(a)) = (\phi^*(f),a)$ and 
	$(m(f\otimes g), a) = (f\otimes g, \Delta(a))$
	(for $f,g\in C^*$ and $a\in C$).
\end{proof}

For a coalgebra $C$ and $f\in C^*$, the map
$f^\perp\colon C\lra C$ is defined by
$f^\perp(c) = \sum_{(c)} (f, c_1) c_2$
where we write $\Delta(c)=\sum_{(c)}c_1\otimes c_2$ by the Sweedler notation.
In other words $f^\perp=(f\otimes\id)\circ\Delta$.
By \eqref{eq:dual_copro}, $f^\perp$ is the adjoint of
$(f\cdot)\colon C^*\lra C^*\,;\,g\mapsto fg$,
i.e.\,$(fg,c) = (g,f^\perp(c))$.
Since the multiplication map $(f\cdot)$ is a $C^*$-module morphism,
by Lemma \ref{theo:comodule_hom} we see that
$f^\perp$ is a $C$-comodule morphism.
Conversely, any $C$-comodule endomorphism on $C$ has the form $f^\perp$:
\begin{lemm}
	\Enumerate{
		\item
			For an algebra $A$,
			if $\psi\colon A\lra A$ is an $A$-module morphism then
			$\psi$ is the multiplication by $\psi(1_A)$.
		\item
			For a coalgebra $C$,
			if $\phi\colon C\lra C$ is a $C$-comodule morphism then
			$\phi = (\phi^*(1_{C^*}))^\perp$.
	}
\end{lemm}
\begin{proof}
	(1) is clear.
	(2) follows from (1),
	Lemma \ref{theo:comodule_hom} and 
	the adjointness of $(\phi^*(1)\cdot)$ and $\phi^*(1)^\perp$.
\end{proof}

\subsection{Symmetric functions}

For basic definitions for symmetric functions,
see for instance \cite[Chapter I]{MR1354144}.

Let $\La$ ($=\La(x)=\La_{K}=\La_K(x)$) be the ring of symmetric functions,
namely the set of all symmetric formal power series 
of bounded degree
in variable $x=(x_1,x_2,\dots)$ 
with coefficients in $K$.
We omit the variable $x$ when no confusion arise.
Let $\wh\La$ be its completion,
consisting of all symmetric formal power series
(with possibly unbounded degree).
The Schur functions $s_\la$ ($\la\in\PP$) are a family of homogeneous symmetric functions
satisfying
$\La=\bigoplus_{\la\in\PP}K s_\la$ 
and
$\wh\La=\prod_{\la\in\PP}K s_\la$.

The {\em Hall inner product} $(\,,)$ is a bilinear form on $\La$
for which 
the Schur functions form an orthonormal basis,
i.e.\,$(s_\la,s_\mu)=\delta_{\la\mu}$.
This is naturally extended to
$(\,,)\colon\wh\La\times\La\lra K$,
whence we can identify
$\wh\La$ with $\La^*$ 
and
$\La$ with $\La^o=\bigoplus_{n\ge 0} \La_n^*$. 
Here $\La_n$ denotes the homogeneous component of $\La$ with degree $n$.

The ring $\La$ is a Hopf algebra with
a product $m\colon\La\otimes\La\lra\La\,;\,f\otimes g\mapsto fg$,
a unit $u\colon K\lra\La\,;\,1\mapsto 1$, 
a coproduct $\Delta\colon\La=\La(x)\lra\La(x,y)\hookrightarrow\La(x)\otimes\La(y)\,;\,f(x)\mapsto f(x,y)$,
a counit $\epsilon\colon\La\lra K\,;\,f\mapsto f(0,0,\dots)$, 
i.e.\,$\epsilon(s_\la)=\delta_{\la\emptyset}$,
and an antipode $S\colon\La\lra\La\,;\,s_\la\mapsto(-1)^{|\la|}s_{\la'}$.
Here $\la'$ denotes the transpose of $\la\in\PP$.
The coincidence between
the coefficients in the Littlewood-Richardson rules
$s_\mu s_\nu=\sum_{\la} c^\la_{\mu\nu}s_\la$ and
$\Delta(s_\la)=\sum_{\mu,\nu} c^\la_{\mu\nu}s_\mu\otimes s_\nu$
implies that
$\La$ is self-dual,
i.e.\,the Hopf structure on $\La^o$ 
via \eqref{eq:dual_copro} and \eqref{eq:dual_pro}
coincides with
one coming from the identification $\La\simeq\La^o$.
Note that $\wh\La\simeq\La^*$ is an algebra 
but not a coalgebra,
since if $f\in\wh\La$ has unbounded degree then $f(x,y)$ may be unable to be written as a finite sum of $f_1(x)f_2(y)$ for $f_1,f_2\in\wh\La$.


For $F\in\wh\La$,
we have linear maps
\Itemize{
	\item
	$(F,-)\colon\La\lra K\,;\,f\mapsto (F,f)$, and
	\item
	$F^\perp\colon\La\lra\La\,;\,f\mapsto\sum (F,f_1)f_2$
}
where 
we put $\Delta(f)=\sum f_1\otimes f_2$ for $f\in\La$ by the Sweedler notation.
By the identification $\wh\La\simeq\La^*$
this notation is the same as given in Section \ref{sect:Prel::Hopf}.
Note that 
\begin{equation}\label{eq:Fperp}
F^\perp = ((F,-)\otimes\id)\circ\Delta = (\id\otimes(F,-))\circ\Delta
\end{equation}
where the second equality is by cocommutativity.
%
%
We also have
\begin{equation}\label{eq:F-}
(F,-) = \epsilon\circ F^\perp
\end{equation}
since
$\epsilon\circ F^\perp 
= \epsilon\circ((F,-)\otimes\id)\circ\Delta 
= ((F,-)\otimes\epsilon)\circ\Delta 
= (F,-)*\epsilon = (F,-)$.
The following lemma is standard:
\begin{lemm}\label{theo:FG}
	For $F,G\in\wh\La$,
	\Enumerate{
		\item 
		$(FG,-) = (F,-)*(G,-)$ where
		$*$ denotes the convolution product on
		$\Hom(\La,K)$.
		\item 
		$(FG)^\perp=G^\perp\circ F^\perp$ $(=F^\perp\circ G^\perp)$.
	}
\end{lemm}


By arguments in Section \ref{sect:Prel::Hopf::comod} we have the following lemmas.

\begin{lemm}\label{theo:Fperp}
	For $F\in\wh\La$,
	\Enumerate{
		\item
		$F^\perp\colon\La\lra\La$ is a $\La$-comodule morphism.
		\item 
		$(F^\perp)^*=(F\cdot)$.
		Namely
		$(FG,f) = (G,F^\perp(f))$ for $G\in\wh\La$ and $f\in\La$.
	}
	
\end{lemm}

\begin{lemm}\label{theo:cohom_mult}
	Let $\phi\colon\La\lra\La$ be a $\La$-comodule morphism.
	Then the dual map $\phi^*\colon\wh\La\lra\wh\La$
	is the multiplication by $\phi^*(1)$.
	Moreover $\phi = (\phi^*(1))^\perp$.
\end{lemm}

\subsubsection{Stable Grothendieck polynomials}

In \cite[Theorem 3.1]{MR1946917}
Buch gave a combinatorial description of
the {\em stable Grothendieck polynomial} $G_\la$
as a generating function of so-called {\em set-valued tableaux}.
We do not review the detail here
and just recall some of its properties:
$G_\la\in\wh\La$ (although $G_\la\notin\La$ if $\la\neq\emptyset$),
$G_\la$ is an infinite linear combination of $\{s_\mu\}_{\mu\in\PP}$
whose lowest degree component is $s_\la$.
Hence
$\wh\La=\prod_{\la\in\PP}K G_\la$,
i.e.\,every element in $\wh\La$ is uniquely written as an infinite linear combination of $G_\la$.
Moreover
the span
$\bigoplus_\la K G_\la$ ($\subset\wh\La$) is a bialgebra,
in particular the expansion of 
the product
$G_\mu G_\nu = \sum_{\la} c^{\la}_{\mu\nu} G_\la$ and
the coproduct
$\Delta(G_\la) = \sum_{\mu,\nu} d^{\la}_{\mu\nu} G_\mu\otimes G_\nu$
are finite.

\subsubsection{Dual stable Grothendieck polynomials}

Next we recall the dual stable Grothendieck polynomial $g_{\la/\mu}$.
For a skew shape $\la/\mu$,
a \textit{reverse plane partition} of shape $\la/\mu$
is a filling of the boxes in $\la/\mu$ with positive integers such that
the numbers are
weakly increasing
in every row and column.

\begin{defi}[{\cite{MR2377012}}]\label{defi:gla}
	For a skew shape $\la/\mu$,
	the {\em dual stable Grothendieck polynomial} $g_{\la/\mu}$ is defined by
	\begin{equation}\label{eq:gdef}
		g_{\la/\mu} = \sum_{T} x^T,
	\end{equation}
	summed over reverse plane partitions $T$ of shape $\la/\mu$,
	where $x^T=\prod_{i} x_i^{T(i)}$
	where $T(i)$ is the number of columns of $T$ that contain $i$.
\end{defi}
	When $\mu=\emptyset$ we write $g_\la=g_{\la/\emptyset}$.
It is shown in \cite{MR2377012} that
$g_{\la/\mu}\in\La$ 
and $g_\la$ has 
the highest degree component $s_{\la}$
and forms a basis of $\La$
that is dual to $G_\la$ via the Hall inner product:
$(G_\la,g_\mu)=\delta_{\la\mu}$.
Hence the product (resp.\,coproduct) structure constants for $\{G_\la\}$
coincide with
the coproduct (resp.\,product) structure constants for $\{g_\la\}$:
$g_\mu g_\nu = \sum_{\la} d^{\la}_{\mu\nu} g_\la$ and
$\Delta(g_\la) = \sum_{\mu,\nu} c^{\la}_{\mu\nu} g_\mu\otimes g_\nu$.

After \cite{dualstablesum}, 
we denote by $I$ the linear map $\La\lra\La$ defined by
$I(g_\la)=\sum_{\mu\subset\la}g_\mu$.
In \cite{dualstablesum} it is shown that $I$ is a ring automorphism,
but we need not use it as we rediscover it in \eqref{eq:I=i*id}.
%
	The inverse map
	is given by
	$I^{-1}(g_\la)=\sum_{\text{$\la/\mu$: rook strip}} (-1)^{|\la/\mu|} g_\mu$.
	Here $\la/\mu$ is called a \emph{rook strip} if any cell of $\la/\mu$ is removable corner of $\la$.

\begin{rema}\label{rema:Gg_geom}
	We recall geometric interpretations of $G_\la$ and $g_\la$.
	Let 
	$\Gr(k,n)$ be the Grassmannian of $k$-dimensional subspaces of $\mathbb{C}^n$,
	$R=(n-k)^k$ the rectangle of shape $(n-k)\times k$,
	and $\mcO_\la$ ($\la\subset R$) the structure sheaves of Schubert varieties of $\Gr(k,n)$.
	The $K$-theory $\Kth$,
	the Grothendieck group of algebraic vector bundles on $\Gr(k,n)$,
	has a basis 
	$\{[\mcO_\la]\}_{\la\subset R}$, 
	and the surjection
	$\bigoplus_{\la\in\PP} \Z G_\la \lra \Kth = \bigoplus_{\la\subset R} \Z [\mcO_\la]$
	that maps $G_\la$ to $[\mcO_\la]$
	(which is considered as $0$ if $\la\not\subset R$)
	is an algebra homomorphism \cite{MR1946917}.

	There is another basis of $\Kth$ consisting of the classes $[\mcI_\la]$
	of ideal sheaves of boundaries of Schubert varieties.
	In \cite[Section 8]{MR1946917} it is shown that
	the bases $\{[\mcO_\la]\}_{\la\subset R}$ and $\{[\mcI_\la]\}_{\la\subset R}$ relates to each other by
	$[\mcO_\la] = \sum_{\la\subset\mu\subset R} [\mcI_\mu]$
	and that they are dual:
	more precisely $([\mcO_\la],[\mcI_{\ti\mu}])=\de_{\la\mu}$ where
	$\ti\mu=(n-k-\mu_k,\cdots,n-k-\mu_1)$ is the rotated complement of $\mu\subset R$
	and the pairing $(\,,)$ is defined by
	$(\alpha,\beta) = \rho_*(\alpha\otimes\beta)$ 
	where
	$\rho_*$
	is the pushforward to a point.
	
	%
	The $K$-homology $\Kho$,
	the Grothendieck group of coherent sheaves,
	is naturally isomorphic to $\Kth$. 
	Lam and Pylyavskyy proved in \cite[Theorem 9.16]{MR2377012} that the surjection
	$\La=\bigoplus_{\la\in\PP} \Z g_\la \lra \Kho = \bigoplus_{\mu\subset R} \Z [\mcI_\mu]$
	that maps $g_\la$ to $[\mcI_{\ti\la}]$
	(which is considered as $0$ if $\la\subset R$)
	identifies the coproduct and product on $\La$ with
	the pushforwards of
	the diagonal embedding map
	and
	the direct sum map.
	
	Since $\mu\subset\la\iff\tilde\mu\supset\tilde\la$, 
	under this identification we see that
	$\sum_{\mu\subset\la}g_\mu\in\La$ corresponds to $[\mcO_{\tilde\la}]\in\Kho$.
\end{rema}

\section{Group-like elements in $\wh\La$ and automorphisms on $\La$}\label{sect:algebraic}

An element $a$ of a coalgebra $(A,\Delta,\epsilon)$ is called \emph{group-like} if $\Delta(a)=a\otimes a$ and $\epsilon(a)=1$.
The set of group-like elements in a Hopf algebra forms a group;
if $a$ and $b$ are group-like then so are $ab$ and $S(a)=a^{-1}$.
If $A$ is a bialgebra and $a\in A$ is group-like
then the multiplication map $L_a\colon A\lra A\,;\,b\mapsto ab$ is a coalgebra morphism
since $\Delta\circ L_a(b) = \Delta(ab) = \Delta(a)\Delta(b) = (a\otimes a)\Delta(b) = (L_a\otimes L_a)\circ\Delta(b)$,
and hence the dual map
$L_a^*\colon A^*\lra A^*$ is an algebra morphism.

Although $\wh\La$ is not a coalgebra,
we shall also say $F\in\wh\La$ is \emph{group-like} if $F(x,y)=F(x)F(y)$ and 
its constant term $F(0)$ is $1$.
Again, here we mean
$F(x)=F(x_1,x_2,\cdots)$,
$F(y)=F(y_1,y_2,\cdots)$,
$F(x,y)=F(x_1,x_2,\cdots, y_1, y_2,\cdots)$ and
$F(0) = F(0,0,\cdots)$.
%
Then these elements satisfy expected properties seen above:

\begin{lemm}
	For group-like elements $F,F'\in\wh\La$,
	\Enumerate{
		\item
			$FF'$ is group-like.
		\item
			$S(F)=F^{-1}$ is group-like.
			Here we extend the antipode 
			$S\colon\La\lra\La\,;\,s_\la\mapsto(-1)^{|\la|}s_{\la'}$ to $\wh\La\lra\wh\La$.
	}
\end{lemm}
\begin{proof}
	\noindent (1)
	By
	$FF'(x,y)=F(x,y)F'(x,y)=F(x)F(y)F'(x)F'(y)=FF'(x)\cdot FF'(y)$ and
	$FF'(0) = F(0)F'(0)=1$.
	
	\noindent (2)
	Write $F=\sum_{\la}A_\la s_\la$ with $A_\la\in K$ (possibly an infinite sum).
	
	Since
	$F(x,y) 
	= \sum_{\la} A_\la s_\la(x,y)
	= \sum_{\la} A_\la \sum_{\mu,\nu} c^{\la}_{\mu\nu} s_\mu(x)s_\nu(y)
	= \sum_{\mu,\nu} \big(\sum_{\la} A_\la c^{\la}_{\mu\nu}\big) s_\mu(x)s_\nu(y)$
	and
	$F(x)F(y) 
	= \big(\sum_{\mu} A_\mu s_\mu(x)\big) \big(\sum_{\nu} A_\nu s_\nu(y)\big)
	= \sum_{\mu,\nu} A_\mu A_\nu s_\mu(x)s_\nu(y)$,
	it follows that
	\begin{equation}\label{eq:grouplike_cond}
		\text{$F=\sum_{\la}A_\la s_\la$ is group-like}
		\iff
		A_\emptyset = 1 \text{ and }
		A_\mu A_\nu = \sum_{\la} A_\la c^{\la}_{\mu\nu} \text{ for } \forall \mu, \nu.
	\end{equation}
	
	Let $F':=S(F)=\sum A_\la (-1)^{|\la|}s_{\la'}$.
	Similarly we can see that
	$F'$ is group-like if and only if
	$A_\emptyset=1$ and 
	$A_\mu A_\nu = \sum_{\la} A_\la c^{\la'}_{\mu'\nu'}$ for any $\mu$, $\nu$.
	Since $c^{\la}_{\mu\nu}=c^{\la'}_{\mu'\nu'}$ it follows that
	$F'$ is group-like.
	
	Since $\Delta(s_\la) = \sum_{\mu,\nu} c^{\la}_{\mu\nu} s_\mu\otimes s_\nu$,
	by applying $m\circ(\id\otimes S)$ we have
	$\sum_{\mu,\nu} (-1)^{|\nu|} c^{\la}_{\mu\nu} s_\mu s_{\nu'} 
		= m\circ(\id\otimes S)\circ\Delta(s_\la)
		= u\circ\epsilon(s_\la) = \de_{\la,\emptyset}$.
	From this and \eqref{eq:grouplike_cond} we have
	$FF' 
		= \big(\sum_{\mu} A_\mu s_\mu\big) \big(\sum_\nu (-1)^{|\nu|} A_\nu s_{\nu'}\big)
		= \sum_{\mu,\nu} (-1)^{|\nu|} A_\mu A_\nu s_\mu s_{\nu'}
		= \sum_{\mu,\nu} \sum_{\la} (-1)^{|\nu|} c^{\la}_{\mu\nu} A_\la  s_\mu s_{\nu'}
		= \sum_{\la} A_\la \big(\sum_{\mu,\nu} (-1)^{|\nu|} c^{\la}_{\mu\nu} s_\mu s_{\nu'}\big)
		= \sum_{\la} A_\la \de_{\la,\emptyset}
		= A_\emptyset = 1$.
	Hence $F' = F^{-1}$.
\end{proof}

\begin{lemm}\label{theo:group-like}
	For $F\in\wh\La$,
	the followings are equivalent.
	\Enumerate{
		\item 
		$F\in\wh\La$ is group-like.
		\item 
		$(F,-)\colon\La\lra K$ is an algebra homomorphism.
		\item 
		$F^\perp\colon\La\lra\La$ is an algebra automorphism.
	}
\end{lemm}
\begin{proof}
	\noindent\underline{$(1)\iff(2)$}
	Again we write $F=\sum_{\la}A_\la s_\la$ with $A_\la\in K$.
	(2) is equivalent to 
	$(F,1)=1$ and
	$(F,s_\mu) (F,s_\nu) = (F,s_\mu s_\nu)$ for any $\mu$, $\nu$,
	which is equivalent to \eqref{eq:grouplike_cond}
	since
	$s_\mu s_\nu = \sum_{\la} c^{\la}_{\mu\nu} s_\la$.
	
	\noindent\underline{$(2)\implies(3)$}
	Since $(F,-)\colon\La\lra K$ and $\id_\La\colon\La\lra\La$ are algebra morphisms,
	$(F,-)\otimes\id_\La\colon\La\otimes\La\lra\La$ is an algebra morphism.
	Since $\Delta\colon\La\lra\La\otimes\La$ is an algebra morphism by the axiom of bialgebras,
	it follows that $F^\perp=((F,-)\otimes\id)\circ\Delta$ is an algebra morphism.
	
	By $S(F) = F^{-1}$ and Lemma \ref{theo:FG} (2) we have
	$S(F)^\perp = (F^\perp)^{-1}$.
	Hence $F^\perp$ is invertible.
	
	\noindent\underline{$(3)\implies(2)$}
	By \eqref{eq:F-} and the axiom of bialgebras that $\epsilon\colon\La\lra K$ is an algebra morphism,
	it follows that $(F,-)=\epsilon\circ F^\perp$ is an algebra morphism.
\end{proof}

\begin{rema}
	There is no group-like element in $\La$ except $1$
	since $f(x,y)=f(x)f(y)$ implies $\deg(f) = \deg(f) + \deg(f)$.
\end{rema}

\subsection{Group-like elements $H(t)$, $E(t)$}

There are well-known generating functions
\begin{equation*}
H(t) = \sum_{i\ge 0} t^i h_i,
\qquad
E(t) = \sum_{i\ge 0} t^i e_i
\end{equation*}
where $t\in K$.
Note that 
\begin{equation}\label{eq:HE}
	H(t)E(-t)=1.
\end{equation}

\begin{lemm}\label{theo:HE:group-like}
	The elements $H(t), E(t) \in \wh\La$ 
	are group-like.
\end{lemm}
\begin{proof}
	By
	$\Delta(h_k)=\sum_{i+j=k}h_i\otimes h_j$,
	we have
	$H(t)(x,y)
		= \sum_{k\ge 0} t^k h_k(x,y)
		= \sum_{k\ge 0} t^k \sum_{i+j=k} h_i(x) h_j(y)
		= \big(\sum_{i\ge 0} t^i h_i(x)\big) \big(\sum_{j\ge 0} t^j h_j(y)\big)
		= \big(H(t)(x)\big) \big(H(t)(y)\big)$.
	The proof for $E(t)$ is similar.
\end{proof}

Define an algebra homomorphism
\[
	\phi_t\colon\La\lra\La\,;\,f(x)\mapsto f(tx),
\]
where $x=(x_1,x_2,\dots)$ and $tx=(tx_1,tx_2,\dots)$.
Clearly $\phi_t$ extends to an algebra homomorphism on $\wh\La$.
If $t$ is invertible then $\phi_t\circ\phi_{t^{-1}}=\id$,
whence $\phi_t$ is an automorphism.
Since $(\phi_t(s_\la),s_\mu) = t^{|\la|} \delta_{\la\mu} = (s_\la, \phi_t(s_\mu))$,
the map $\phi_t$ is self-adjoint.
Note that 
\begin{equation}\label{eq:Ht}
H(t)=\phi_t(H(1))\quad\text{ and }\quad E(t)=\phi_t(E(1)).
\end{equation}

\subsection{corresponding algebra morphisms}

By Lemma \ref{theo:group-like} and \ref{theo:HE:group-like} we have
algebra homomorphisms
\[
(H(t),-)\colon\La\lra K
\qquad\text{and}\qquad
(E(t),-)\colon\La\lra K,
\]
and algebra automorphisms
\[
H(t)^\perp\colon\La\lra\La
\qquad\text{and}\qquad
E(t)^\perp\colon\La\lra\La.
\]

\begin{lemm}\label{theo:(Ht,-)}
	$(H(t),-) = (H(1),-)\circ\phi_t$ and
	$(E(t),-) = (E(1),-)\circ\phi_t$.
\end{lemm}
\begin{proof}
	Since $\phi_t$ is self-adjoint,
	for $f\in\La$
	it follows that
	$(H(t),f) = (\phi_t(H(1)),f) = (H(1),\phi_t(f))$ and
	similarly $(E(t),f) = (E(1),\phi_t(f))$.
\end{proof}

\begin{lemm}\label{theo:Ht_perp}
	$\phi_t\circ H(t)^\perp=H(1)^\perp\circ \phi_t$ and
	$\phi_t\circ E(t)^\perp=E(1)^\perp\circ \phi_t$.
	
	Hence
	$H(t)^\perp=\phi_t^{-1}\circ H(1)^\perp\circ \phi_t$ and
	$E(t)^\perp=\phi_t^{-1}\circ E(1)^\perp\circ \phi_t$
	when $t$ is invertible.
\end{lemm}
\begin{proof}
	Since $\phi_t$ is a self-adjoint algebra automorphism,
	for $F,G\in\wh\La$ and $f\in\La$ we have
	$(F, \phi_t\circ(\phi_t(G))^\perp(f)) 
	= (\phi_t(F), (\phi_t(G))^\perp(f)) 
	= (\phi_t(G)\phi_t(F), f) 
	= (\phi_t(GF), f) 
	= (G F, \phi_t(f)) 
	= (F, G^\perp\circ\phi_t(f))$,
	whence
	$\phi_t\circ(\phi_t(G))^\perp = G^\perp\circ\phi_t$.
	In particular
	$\phi_t\circ H(t)^\perp = H(1)^\perp\circ\phi_t$
	and
	$\phi_t\circ E(t)^\perp = E(1)^\perp\circ\phi_t$
	by \eqref{eq:Ht}.
\end{proof}

\begin{lemm}\label{theo:HE_Hall_perp}
	\Enumerate{
		\item 
			$(H(t),-)*(E(-t),-) = \epsilon$,
			where $\epsilon\colon\La\lra K$ is the counit.
		\item 
			$H(t)^\perp \circ E(-t)^\perp = \id_\La$.
	}
\end{lemm}
\begin{proof}
	\noindent (1)
	By \eqref{eq:HE}, Lemma \ref{theo:FG}, and the fact that 
	the counit is the identity element with respect to the convolution product $*$.
%
	(2)
	By \eqref{eq:HE} and Lemma \ref{theo:FG}.
\end{proof}

\section{Proof of $I=H(1)^\perp$ and $I^{-1}=E(-1)^\perp$}\label{sect:I}


Define a ring homomorphism $i\colon \La\lra K$ by 
the substitution
\[
i\colon f\mapsto f(1) = f(1,0,0,\dots).
\]
We shall show in Theorem \ref{theo:H(1)_overall}
that $i=(H(1),-)$ and $I=(i\otimes\id)\circ\Delta=H(1)^\perp$.
We start with the following observation.

\begin{prop}\label{theo:i}
	$i(g_{\la/\mu})=1$ for any skew shape $\la/\mu$,
	and in particular $i(g_{\la})=1$ for any $\la\in\PP$.
\end{prop}

\begin{proof}
	By the definition \eqref{eq:gdef} of $g_{\la/\mu}$,
	$i(g_{\la/\mu})$ is the number of reverse plane partitions on $\la/\mu$
	filled with $1$.
	Clearly there is exactly one such filling.
\end{proof}

\begin{rema}\label{theo:i:hep}
	It is straightforward to check that
	$i(h_k)=1$ for any $k\ge 0$,
	$i(e_k)=0$ for any $k\ge 2$, and
	$i(p_k)=1$ for any $k\ge 1$.
\end{rema}

Since $\{g_\la\}_\la$ is a basis of $\La$ and $\{G_\la\}_\la$ is their dual,
from Proposition \ref{theo:i} we have

\begin{equation}\label{eq:i=(G,-)}
	i=\Big(\sum_{\la\in\PP}G_\la, -\Big).
\end{equation}

Another corollary of the proposition above is
formulas on the structure constants in
$g_\mu g_\nu = \sum_\la d^{\la}_{\mu\nu} g_\la$ and
$g_{\la/\mu} = \sum_{\nu} c^{\la}_{\mu\nu} g_\nu$:
\begin{coro}\label{theo:cd}
	\Enumerate{
		\item
			For any $\mu,\nu\in\PP$ we have
			$\sum_\la d^{\la}_{\mu\nu} = 1$.
		\item
			For any $\la,\mu\in\PP$ we have
			$\sum_\nu c^{\la}_{\mu\nu} = 1$.
	}
\end{coro}

\subsection{Description of $I$ as a substitution}\label{sect:I:f(1,x)}

Next we give another description of the map 
$I\colon g_\la \mapsto \sum_{\mu\subset\la}g_\mu$.

For a skew shape $\la/\mu$ and 
a totally ordered set $X$ called {\em alphabets} (most commonly $\{1,2,3,\dots\}$),
we shall denote by $\RPP(\la/\mu,X)$ the set of reverse plane partition
of shape $\la/\mu$ where each box is filled with an element of $X$.
The expression \eqref{eq:gdef} of $g_{\la/\mu}$ as a generating function of reverse plane partitions
implies
\begin{equation}\label{eq:copro}
\Delta(g_{\la/\mu}) 
= \sum_{\mu\subset\nu\subset\la} g_{\la/\nu} \otimes g_{\nu/\mu},
\end{equation}
since we have a natural bijection between
$\RPP(\la/\mu, \{1,2,\cdots,1',2',\dots\})$ and $\bigsqcup_{\mu\subset\nu\subset\la} \RPP(\nu/\mu, \{1,2,\cdots\})\times\RPP(\la/\nu, \{1',2',\cdots\})$
where $1<2<\cdots<1'<2'<\cdots$.
Since 
$(i\otimes\id)\circ\Delta\colon 
f(x)\mapsto 
f(1,x)$
where $f(x)=f(x_1,x_2,\dots)$ and $f(1,x)=f(1,x_1,x_2,\dots)$,
by applying $i\otimes\id$ to \eqref{eq:copro}
we have
\begin{equation}\label{eq:g1x}
g_{\la/\mu}(1,x)
= \sum_{\mu\subset\nu\subset\la} 
i(g_{\la/\nu}) g_{\nu/\mu}(x)
= \sum_{\mu\subset\nu\subset\la} g_{\nu/\mu}(x),
\end{equation}
where the last equation is by Proposition \ref{theo:i}.
Similarly, by applying $(\id\otimes i)\circ\Delta$ to \eqref{eq:copro} we have
\begin{equation}\label{eq:gx1}
g_{\la/\mu}(x,1)
= \sum_{\mu\subset\nu\subset\la} 
g_{\la/\nu}(x) i(g_{\nu/\mu})
= \sum_{\mu\subset\nu\subset\la} g_{\la/\nu}(x).
\end{equation}
Setting $\mu=\emptyset$ in \eqref{eq:g1x},
for any $\la\in\PP$ we have
\[
g_{\la}(1,x)
= \sum_{\nu\subset\la} g_{\nu}(x) \ \Big(=I(g_\la(x))\Big).
\]
Since $\{g_\la\}_\la$ form a basis of $\La$,
this implies
\begin{prop}\label{theo:I(f)}
	For any $f\in\La$ we have
	\begin{equation}\label{eq:I=i*id}
		I(f(x)) = f(1,x),
	\end{equation}
	or equivalently
	\begin{equation}\label{eq:I_i}
		I=(i\otimes\id)\circ\Delta. 
	\end{equation}
\end{prop}

In particular \eqref{eq:I=i*id} recovers that $I\colon\La\lra\La$ is a ring homomorphism,
and the bijectivity follows from the fact that
the transition matrix between $g_\la$ and 
$I(g_\la)=\sum_{\mu\subset\la}g_\mu$
is unitriangular.

Moreover, \eqref{eq:g1x}, \eqref{eq:gx1} and \eqref{eq:I=i*id} imply
	that for any skew shape $\la/\mu$
	\begin{equation}\label{eq:I_skew}
	I(g_{\la/\mu}) 
	= \sum_{\mu\subset\nu\subset\la} g_{\nu/\mu}
	= \sum_{\mu\subset\nu\subset\la} g_{\la/\nu}.
	\end{equation}

Besides,
by \eqref{eq:Fperp}, \eqref{eq:I_i} and Corollary \ref{eq:i=(G,-)} we have
\begin{equation}\label{eq:Iasperp}
I=\Big(\sum_{\la\in\PP} G_\la\Big)^\perp.
\end{equation}
%
By \eqref{eq:Iasperp} and Lemma \ref{theo:Fperp},
$I$ is a $\La$-comodule morphism.
Since $I$ is bijective 
we have 
%
\begin{prop}\label{theo:copro}
	$I\colon\La\lra\La$ is a $\La$-comodule automorphism.
\end{prop}

\begin{rema}
	\noindent (1)
	We can prove Proposition \ref{theo:copro} by directly showing
	the equality
	\begin{equation}\label{eq:DeltaI}
	\Delta\circ I = (I\otimes \mathrm{id}_{\La})\circ\Delta 
	= (\mathrm{id}_{\La}\otimes I)\circ\Delta.
	\end{equation}
	Indeed,
	by \eqref{eq:I=i*id}
	each side of \eqref{eq:DeltaI} maps
	$f(x)\in\La(x)$ to $f(1,x,y)\in\La(x)\otimes\La(y)$.
	
	\noindent (2)
	Since $(G_\mu,g_\nu)=\de_{\mu\nu}$ and 
	$\Delta(g_\la) = \sum_{\nu} g_{\nu}\otimes g_{\la/\nu}$
	(by \eqref{eq:copro}),
	we have $G_\mu^\perp(g_\la) = g_{\la/\mu}$.
	Here we consider $g_{\la/\mu}=0$ if $\mu\not\subset\la$.
	Since $F^\perp$ (for $F\in\wh\La$) commute each other,
	$I$ commutes with $G_\mu^\perp$.
	Hence, 
	by applying $G_\mu^\perp$ to the equation $I(g_\la)=\sum_{\nu\subset\la}g_\nu$
	we have
	\[
	I(g_{\la/\mu})
	= I(G_\mu^\perp(g_\la))
	= G_\mu^\perp(I(g_\la))
	= G_\mu^\perp\Big(\sum_{\nu\subset\la} g_\nu\Big)
	= \sum_{\nu\subset\la} G_\mu^\perp(g_\nu)
	= \sum_{\nu\subset\la} g_{\nu/\mu},
	\]
	re-proving the first equation in \eqref{eq:I_skew}.
	Similarly, by applying $I=\sum_{\nu}G_\nu^\perp$ to $g_\la$
	we get a special case of the second equation of \eqref{eq:I_skew}
	\[
		I(g_\la) = \sum_{\nu}G_\nu^\perp(g_\la) = \sum_{\nu} g_{\la/\nu}.
	\]
\end{rema}

\subsection{Dual map $I^*$ and proof of $I^*(1)=H(1)$}\label{sect:I:H1_perp}

By Lemma \ref{theo:cohom_mult} and Proposition \ref{theo:copro} 
\Itemize{
	\item
		$I^*$ is the multiplication by $I^*(1)$,
	\item
		$I=I^*(1)^\perp$,
}
and similarly
\Itemize{
	\item
		$(I^{-1})^*$ is the multiplication by $(I^{-1})^*(1)$,
	\item
		$I^{-1}=(I^{-1})^*(1)^\perp$.
}

Since 
$(G_\la,g_\mu)=\delta_{\la\mu}$,
the maps
$I\colon\La\lra\La\,;\,g_\la\mapsto \sum_{\mu\subset\la}g_\mu$ and
$I^{-1}\colon g_\la\mapsto \sum_{\text{$\la/\mu$: rook strip}} (-1)^{|\la/\mu|} g_\mu$ 
induce
the dual maps
\begin{align}
	I^*&\colon\wh\La\lra\wh\La\,;\,G_\la\mapsto \sum_{\la\subset\mu} G_\mu \label{eq:I*} \\
	(I^{-1})^*&\colon\wh\La\lra\wh\La\,;\,G_\la\mapsto\sum_{\text{$\mu/\la$: rook strip}} (-1)^{|\mu/\la|} G_\mu. \label{eq:I*-1}
\end{align}
By \eqref{eq:I*} we have $I^*(1)=\sum_{\mu\in\PP}G_\mu$.
Similarly $(I^{-1})^*(1) = 1-G_1$ by \eqref{eq:I*-1}.
In particular $(1-G_1)\sum_{\mu}G_\mu = 1$.
Since $G_{1}=e_1-e_2+e_3-\cdots$ it follows that $E(-1)=1-G_1$,
whence 
\begin{equation}\label{eq:I-1:detail}
	(I^{-1})^* = (E(-1)\cdot) \qquad\text{and}\qquad I^{-1} = E(-1)^\perp.
\end{equation}
Hence,
it follows from $H(1)E(-1)=1$ that 
\begin{equation}\label{eq:I:detail}
	I^* = (H(1)\cdot),
	\qquad
	\sum_{\la\in\PP}G_\la \,\big(= I^*(1)\,\big)\,= H(1),
	\qquad\text{and}\qquad
	I = H(1)^\perp.
\end{equation}

We have proved the following theorems so far:

\begin{theo}\label{theo:H(1)_overall}
	\Enumerate{
		\item 
			$H(1)=\sum_{i\ge 0} h_i = \sum_{\la\in\PP} G_\la$.
		\item \label{item:H(1):i}
			The ring homomorphism $i\colon\La\lra K$ defined by
			$i(f) = f(1,0,0,\dots)$ satisfies
			\Enumerate{
				\item
					$i = (H(1),-)$.
				\item 
					$i(g_{\la/\mu}) = 1$ for any skew shape $\la/\mu$.
			}
		\item 
			The linear map $I\colon\La\lra\La$ defined by
			$I(g_\la) = \sum_{\mu\subset\la} g_\mu$ 
			is a ring automorphism that satisfies
					\[I = H(1)^\perp 
						= \sum_{\la\in\PP} G_\la^\perp
						= \sum_{i\ge 0} h_i^\perp
						= \big(f(x_1,x_2,\dots)\mapsto f(1,x_1,x_2,\dots)\big)
					\]
			and
			\[
					I(g_{\la/\mu})
						= \sum_{\mu\subset\nu\subset\la} g_{\nu/\mu}
						= \sum_{\mu\subset\nu\subset\la} g_{\la/\nu}
			\]
			for any skew shape $\la/\mu$.
			In particular for any $\la\in\PP$
			\[
				I(g_{\la})
				= \sum_{\nu\subset\la} g_{\nu}
				= \sum_{\nu\subset\la} g_{\la/\nu}.
			\]
	}
\end{theo}

\begin{theo}\label{theo:E1:overall}
	\Enumerate{
		\item 
			$E(-1) = \sum_{i\ge 0}(-1)^i e_i = 1-G_1$.
		\item 
			The map $I^{-1}$ satisfies
			\begin{equation*}
				I^{-1} 
				= E(-1)^\perp
				= \id_\La - G_{1}^\perp
				= \sum_{i\ge 0}(-1)^i e_i^\perp.
			\end{equation*}
	}
\end{theo}

We take a closer look at $(E(-1),-)$ and $E(-1)^\perp$
in Section \ref{sect:E(-1)}.
In Example \ref{exam:I_skew}
we see an example for the identity \eqref{eq:I_skew}.

\begin{rema}
	By \eqref{eq:I*}, \eqref{eq:I*-1}, \eqref{eq:I-1:detail} and \eqref{eq:I:detail}
	we rediscover the formulas
	\begin{align}
	\bigg(I^*(G_\la)=\bigg)\quad
	\big(\sum_{\mu\in\PP}G_\mu\big) G_\la
	&= \sum_{\la\subset\mu} G_\mu
	\label{eq:I*(1)} \\
	\bigg((I^{-1})^*(G_\la)=\bigg)\quad
	(1 - G_{1}) G_\la
	&= \sum_{\text{$\mu/\la$: rook strip}} (-1)^{|\mu/\la|} G_\mu,
	\label{eq:D*(1)}
	\end{align}
	which appeared in \cite[Section 8]{MR1946917}.
	These two are generalized in \eqref{eq:Ht:G} and \eqref{eq:Et:G}.
\end{rema}

\subsection{Description of $(E(-1),-)$ and $I^{-1} = E(-1)^\perp$}
\label{sect:E(-1)}

\begin{prop}\label{theo:(E1,-)_detail}
	The ring homomorphism $(E(-1),-)\colon\La\lra K$ satisfies
		$(E(-1),g_{\la/\mu}) = (-1)^{|\la/\mu|}$
		if $\la/\mu$ is a rook strip,
		and 
		$(E(-1),g_{\la/\mu}) = 0$ otherwise.

		In particular 
		$(E(-1),g_{\emptyset}) = 1$,
		$(E(-1),g_{(1)}) = -1$, and
		$(E(-1),g_{\la}) = 0$
		for any $\la\in\PP$ with $|\la|>1$.
\end{prop}
\begin{proof}
	Since $(H(1),-)*(E(-1),-)=\epsilon$ (Lemma \ref{theo:HE_Hall_perp}),
	applying $(H(1),-)\otimes(E(-1),-)$ to \eqref{eq:copro} we have
	for any skew shape $\la/\mu$ that
	\[
	\delta_{\la\mu} 
	= \epsilon(g_{\la/\mu})
	= \sum_{\mu\subset\nu\subset\la} (H(1),g_{\la/\nu}) (E(-1),g_{\nu/\mu})
	= \sum_{\mu\subset\nu\subset\la} (E(-1),g_{\nu/\mu}).
	\]
	Here the last equality is by
	Theorem \ref{theo:H(1)_overall} \ref{item:H(1):i}.
	Hence $(E(-1),g_{\nu/\mu})$ is equal to 
	the value of the M\"obius function $\mu_{\PP}(\mu,\nu)$,
	which is $(-1)^{|\nu/\mu|}$ if $\nu/\mu$ is a rook strip and $0$ otherwise.
\end{proof}

\begin{rema}\label{theo:E1:hep}
	It is easy to check that
	$(E(-1), e_i) = (-1)^i$ for $i\ge 0$, 
	$(E(-1), h_i) = 0$ for $i \ge 2$, and
	$(E(-1), p_i) = -1$ for $i\ge 1$,
	from Remark \ref{theo:i:hep} and the fact that
	$(H(1),-)*(E(-1),-)=\epsilon$.
\end{rema}

\begin{rema}
	Unlike Proposition \ref{theo:H(1)_overall} \ref{item:H(1):i},
	there is no $a_1,a_2,\dots\in\mathbb{R}$ such that
	$(E(-1),f) = f(a_1,a_2,\dots)$,
	since
	such numbers should satisfy $-1 = (E(-1),p_2) = a_1^2+a_2^2+\cdots$.
\end{rema}

Now we give a description of $E(-1)^\perp=I^{-1}$. 
\begin{prop}
	The ring automorphism $E(-1)^\perp=I^{-1}\colon\La\lra\La$ satisfies
		\[
		I^{-1}(g_{\la/\mu}) 
		= \sum_{\substack{\mu\subset\nu\subset\la \\ \la/\nu\text{: rook strip}}}
		(-1)^{|\la/\nu|} g_{\nu/\mu}
		= \sum_{\substack{\mu\subset\nu\subset\la \\ \nu/\mu\text{: rook strip}}}
		(-1)^{|\nu/\mu|} g_{\la/\mu}.
		\]
		In particular, when $\mu=\emptyset$ we have
		\begin{equation}\label{eq:Iinv:g}
			I^{-1}(g_{\la}) 
			= \sum_{\la/\nu\text{: rook strip}}
				(-1)^{|\la/\nu|} g_{\nu}
			= \begin{cases}
				g_{\la} - g_{\la/(1)} & \text{if $\la\neq\emptyset$}, \\
				1 & \text{if $\la=\emptyset$}.
			\end{cases}
		\end{equation}
\end{prop}
\begin{proof}
	By \eqref{eq:Fperp}, \eqref{eq:copro} and 
	Proposition \ref{theo:(E1,-)_detail}.
\end{proof}

\section{Description of $H(t)$, $(H(t),-)$ and $H(t)^\perp$}\label{sect:Ht}

Now we give a description for the map $(H(t),-)$.
Let $c(\la/\mu)$ denote the number of columns in the skew shape $\la/\mu$.
\begin{prop}\label{theo:(Ht,-)_detail}
	The algebra homomorphism $(H(t),-)\colon\La\lra K$ satisfies
	\begin{flalign*}
	&\mathrm{(1)} & (H(t), f) &= f(t,0,0,\cdots) & &\text{for any $f\in\La$}, \\
	&\mathrm{(2)} & (H(t),g_{\la/\mu}) &= t^{c(\la/\mu)}  & &\text{for any skew shape $\la/\mu$,}\\
	&\phantom{\mathrm{(b)}}\text{and in particular} & (H(t),g_{\la}) &= t^{c(\la)} & &\text{for any $\la\in\PP$.}
	\end{flalign*}
\end{prop}
\begin{proof}
	(1) follows from 
	Lemma \ref{theo:(Ht,-)} and 
	Theorem \ref{theo:H(1)_overall} \ref{item:H(1):i}.
	
	\noindent
	(2) As we argued in the proof of Proposition \ref{theo:i} (2),
	there is exactly one reverse plane partition
	of shape $\la/\mu$ filled with the alphabet $1$,
	whose weight is $x_1^{c(\la/\mu)}$.
	Hence $(H(t),g_{\la/\mu}) = g_{\la/\mu}(t,0,0,\cdots) = t^{c(\la/\mu)}$.
\end{proof}

\begin{rema}
	It is straightforward to check from $(H(t),f) = f(t,0,0,\dots)$ that
	$(H(t), h_i) = t^i$ for $i\ge 0$, 
	$(H(t), e_i) = 0$ for $i \ge 2$, and
	$(H(t), p_i) = t^i$ for $i\ge 1$.
\end{rema}

With the proposition above, Corollary \ref{theo:cd} is generalized as follows:
\begin{coro}
	\Enumerate{
		\item
		For any $\mu,\nu\in\PP$, we have
		$t^{c(\mu)+c(\nu)} = \sum_\la d^{\la}_{\mu\nu} t^{c(\la)}$.
		\item
		For any $\la,\mu\in\PP$, we have
		$t^{c(\la/\mu)} = \sum_\nu c^{\la}_{\mu\nu} t^{c(\nu)}$.
	}
\end{coro}

\begin{prop}\label{theo:Ht_perp:detail}
	The algebra automorphism $H(t)^\perp\colon\La\lra\La$ satisfies
	\begin{flalign*}
	&\mathrm{(1)} & H(t)^\perp(f(x_1,x_2,\dots)) 
	&= f(t,x_1,x_2,\dots)
	& &\text{for any $f\in\La$}, \\
	&\mathrm{(2)} & H(t)^\perp(g_{\la/\mu})
	&= \sum_{\mu\subset\nu\subset\la} t^{c(\la/\nu)} g_{\nu/\mu}
	= \sum_{\mu\subset\nu\subset\la} t^{c(\nu/\mu)} g_{\la/\nu} 
	& &\text{for any $\mu\subset\la$,} \\
	&\phantom{\mathrm{(2)}}\text{and in particular} &
	H(t)^\perp(g_{\la})
	&= \sum_{\nu\subset\la} t^{c(\la/\nu)} g_{\nu}
	= \sum_{\nu\subset\la} t^{c(\nu)} g_{\la/\nu}
	& &\text{for any $\la\in\PP$.}
	\end{flalign*}
\end{prop}
\begin{proof}
	Since
	$H(t)^\perp 
	= ((H(t),-)\otimes\id)\circ\Delta 
	= (\id\otimes(H(t),-))\circ\Delta$,
	(1)
	follows from
	Proposition \ref{theo:(Ht,-)_detail} 
	(1),
	and
	(2)
	is obtained by
	applying 
	$(H(t),-)\otimes\id)$ and
	$\id\otimes(H(t),-))$
	to \eqref{eq:copro} and using
	Proposition \ref{theo:(Ht,-)_detail} 
	(2).
\end{proof}

We also give the expansion of $H(t)$ using $G_\la$.

\begin{prop}\label{theo:Ht}
	The element $H(t)=\sum_{i\ge 0}t^i h_i\in\wh\La$ satisfies
	\begin{align*}
		H(t) G_\la &= \sum_{\la\subset\mu} t^{c(\mu/\la)} G_\mu.
	\end{align*}
	In particular, setting $\la=\emptyset$ we have
	\begin{equation}\label{eq:Ht:G}
		H(t) = \sum_{\mu\in\PP} t^{c(\mu)} G_\mu 
		\qquad\text{and hence}\qquad
		\Big(\sum_{\mu\in\PP} t^{c(\mu)} G_\mu\Big) G_\la
			= \sum_{\la\subset\mu} t^{c(\mu/\la)} G_\mu.
	\end{equation}
\end{prop}
\begin{proof}
	In Proposition \ref{theo:Ht_perp:detail} 
	(2)
	it is shown that 
	$H(t)^\perp(g_{\la})
	= \sum_{\mu\subset\la} t^{c(\la/\mu)} g_{\mu}$,
	from which
	we have
	$(H(t)^\perp)^*(G_\la) = \sum_{\la\subset\mu} t^{c(\mu/\la)} G_\mu$.
	By Lemma \ref{theo:Fperp} (2) and \ref{theo:cohom_mult} 
	we see that
	$(H(t)^\perp)^*$ is the multiplication by
	$(H(t)^\perp)^*(1)=H(t)$.
	Hence the proof is done.
\end{proof}

\section{Description of $E(t)$, $(E(t),-)$ and $E(t)^\perp$}\label{sect:Et}

\begin{prop}\label{theo:(Et,-)_detail}
	The ring homomorphism $(E(t),-)\colon\La\lra K$ satisfies
	\[
		(E(t),g_{\la/\mu}) = 
			\begin{cases}
					t^{c(\la/\mu)} 
					(t+1)^{|\la/\mu|-c(\la/\mu)}
				& \text{if $\la/\mu$ is a vertical strip}, \\
				0 & \text{otherwise}
			\end{cases}
	\]
	for any skew shape $\la/\mu$.
	In particular,
	for any $\la\in\PP$,
	\[
		(E(t),g_{\la}) = 
			\begin{cases}
				1 & \text{if $\la=\emptyset$}, \\
				t (t+1)^{n-1}
				& \text{if $\la=(1^n)$ $(n\ge 1)$}, \\
				0 & \text{otherwise}.
			\end{cases}
	\]
\end{prop}

\begin{rema}
	By Lemma \ref{theo:(Ht,-)} and
	Proposition \ref{theo:E1:hep}
	it follows that
	$(E(t), e_i) = t^i$ for $i\ge 0$, 
	$(E(t), h_i) = 0$ for $i \ge 2$, and
	$(E(t), p_i) = (-1)^{i-1} t^i$ for $i\ge 1$.
\end{rema}

Before proving Proposition \ref{theo:(Et,-)_detail},
we give as its corollaries
descriptions for 
the element $E(t)$ 
and the map $E(t)^\perp$. 

\begin{prop}\label{theo:Et_perp:detail}
	The ring automorphism $E(t)^\perp\colon\La\lra\La$ satisfies
		\begin{align*}
			E(t)^\perp(g_{\la/\mu})
			&= \sum_{\substack{\mu\subset\nu\subset\la \\ \la/\nu\text{: vertical strip}}} 
			t^{c(\la/\nu)} 
			(t+1)^{|\la/\nu|-c(\la/\nu)}
			g_{\nu/\mu} \\
			&= \sum_{\substack{\mu\subset\nu\subset\la \\ \nu/\mu\text{: vertical strip}}}
			t^{c(\nu/\mu)} 
			(t+1)^{|\nu/\mu|-c(\nu/\mu)}
			g_{\la/\nu}
		\end{align*}
		for any skew shape $\la/\mu$.
		In particular,
		for any $\la\in\PP$,
		\begin{align}
			E(t)^\perp(g_{\la})
			&= \sum_{\substack{\nu\subset\la \\ \la/\nu\text{: vertical strip}}} 
			t^{c(\la/\nu)} 
			(t+1)^{|\la/\nu|-c(\la/\nu)}
			g_{\nu} \label{eq:Et_perp:g} \\
			&= 
			\begin{cases}
				g_\la 
				+ \sum_{k=1}^{l(\la)}
				t (t+1)^{k-1}
				g_{\la/(1^k)}
				& \text{if $\la\neq\emptyset$}, \\
				g_\emptyset & \text{if $\la=\emptyset$}.
			\end{cases} \notag
		\end{align}
\end{prop}
\begin{proof}
	Proved similarly to 
	Proposition \ref{theo:Ht_perp:detail} 
	(2),
	with Proposition \ref{theo:(Et,-)_detail} in hand.
\end{proof}

\begin{prop}\label{theo:Et}
	The element $E(t)=\sum_{i\ge 0}t^i e_i\in\wh\La$ satisfies
	\begin{align*}
	E(t) G_\la &= \sum_{\mu/\la\text{: vertical strip}} 
	t^{c(\mu/\la)} (t+1)^{|\mu/\la|-c(\mu/\la)} G_\mu.
	\end{align*}
	In particular, setting $\la=\emptyset$ we have
	\[
		E(t) = 1 + \sum_{n\ge 1} t(t+1)^{n-1} G_{(1^n)},
	\]
	and hence
	\begin{equation}\label{eq:Et:G}
		\Big(1 + \sum_{n\ge 1} t(t+1)^{n-1} G_{(1^n)}\Big) G_\la
		= \sum_{\mu/\la\text{: vertical strip}} 
		t^{c(\mu/\la)} (t+1)^{|\mu/\la|-c(\mu/\la)} G_\mu.
	\end{equation}
\end{prop}
\begin{proof}
	By \eqref{eq:Et_perp:g} we have
	\[
		(E(t)^\perp)^*(G_{\nu})
			= \sum_{\la/\nu\text{: vertical strip}}
			t^{c(\la/\nu)} 
			(t+1)^{|\la/\nu|-c(\la/\nu)}
			G_{\la}.
		\]
	By Lemma \ref{theo:Fperp} (2) and \ref{theo:cohom_mult} 
	we see that
	$(E(t)^\perp)^*$ is the multiplication by
	$(E(t)^\perp)^*(1)=E(t)$.
	Hence the proof is done.
\end{proof}

\subsection{Proof of Proposition \ref{theo:(Et,-)_detail}}

	We recall the \emph{incidence algebras}
	(see \cite[Chapter 3.6]{MR2868112} for details).
	Let $\mathrm{Int}(\PP) = \{(\mu,\la)\in\PP\times\PP\mid\mu\subset\la\}$,
	consisting of all comparable (ordered) pairs in $\PP$
	(or equivalently all skew shapes, 
	by identifying $(\mu,\la)$ with $\la/\mu$).
	The \emph{incidence algebra} $I(\PP)=I(\PP,K)$ is the algebra of all functions
	$f\colon\mathrm{Int}(\PP)\lra K$
	where multiplication is defined by the convolution
	\begin{equation}\label{eq:IP_fg}
		(fg)(\mu,\la) = \sum_{\mu\subset\nu\subset\la} f(\mu,\nu) g(\nu,\la).
	\end{equation}
	Then $I(\PP,K)$ is an associative algebra with two-sided identity
	$\delta := ((\mu,\la)\mapsto \delta_{\mu\la})$.

	A linear function $f\colon\La\lra K$ can be considered as an element of $I(\PP,K)$
	by setting $f(\mu,\la)=f(g_{\la/\mu})$.
	Then the convolution product $*$ on $\Hom(\La,K)$ coincides with
	the multiplication on $I(\PP)$ due to \eqref{eq:copro},
	i.e.\,this inclusion $\Hom(\La,K)\lra I(\PP)$ is as algebras.%
	\footnote{
		Since $\Delta(s_{\la/\mu}) = \sum_{\nu}s_{\la/\nu}\otimes s_{\nu/\mu}$,
		by setting $f(\mu,\la)=f(s_{\la/\mu})$
		we can obtain another algebra inclusion,
		although we do not need it.
	}
	Note that the counit $\epsilon\in\Hom(\La,K)$ is mapped to $\delta\in I(\PP)$.

\begin{proof}[Proof of Proposition \ref{theo:(Et,-)_detail}]
	Define $i_t,j_t\in I(\PP)$ by
	\[
	i_t(\mu,\la) = t^{c(\la/\mu)}
	\]
	and
	\[
	j_t (\mu,\la) =
	\begin{cases}
	(-1)^{|\la/\mu|} 
	t^{c(\la/\mu)} 
	(t-1)^{|\la/\mu|-c(\la/\mu)}
	& \text{if $\la/\mu$ is a vertical strip}, \\
	0 & \text{otherwise}.
	\end{cases}
	\]
	
	By Proposition \ref{theo:(Ht,-)_detail} (2) 
	$(H(t),-)\in\Hom(\La,K)$ corresponds to $i_t\in I(\PP)$.
	Since $(H(t),-)*(E(-t),-)=\epsilon$,
	it suffices to show that $i_t j_t = \delta$
	in order to prove that $(E(-t),-)$ corresponds to $j_t$,
	whence 
	Proposition \ref{theo:(Et,-)_detail}
	follows
	by replacing $t$ with $-t$.
	
	By the definitions of $i_t$ and $j_t$ and \eqref{eq:IP_fg}
	\begin{equation}\label{eq:itjt}
	(i_t j_t)(\mu,\la)
	= \sum_{\substack{\mu\subset\nu\subset\la \\ \la/\nu\text{: vertical strip}}}
	t^{c(\nu/\mu)}
	(-1)^{|\la/\nu|} 
	t^{c(\la/\nu)} 
	(t-1)^{|\la/\nu|-c(\la/\nu)}.
	\end{equation}
	
	We need to show that the value of the right-hand side of \eqref{eq:itjt} is $\delta_{\mu\la}$.
	It is clear that if $\mu=\la$ then the value of \eqref{eq:itjt} is $1$.
	Assume $\mu\subsetneq\la$.
	Since the value of \eqref{eq:itjt} is invariant under 
	removal of empty rows in the skew shape $\la/\mu$,
	we can assume there is a box in the first row of $\la/\mu$,
	i.e.\,$\la_1 > \mu_1$.
	Let $p$ be the index of the rightmost column of $\la$,
	i.e.\,$\la_1=p$.
	Note that $\la'_p>0=\mu'_p$.
	Let $\ti\la$ be the partition obtained by removing the $p$-th (rightmost) column of $\la$,
	i.e.\,$\ti\la_i = \min(\la_i,p-1)$.
	(Note: in the figure below and hereafter we display Young diagrams in the French notation.)
	
	\[
	\tikzitem[0.4]{
		\draw (5,0) -| (9,3) -| (7,5) -| (4,7) -| (0,4) -| (2,2) -| (5,0);
		\draw[thick, loosely dotted] (5,0) -| (0,4);
		
		\draw (9,0) to [out=70, in=-70] node [right]{$\la'_p$}  +(0,3) ;
		\draw [thick, dotted] (8.7,-0.2) -- +(0,-1) node [below]{$p=\la_1$};
		
		\node at (4.5,3.5) {$\la/\mu$};
		\node at (2.5,1) {$\mu$};
		
	}
	\]
	
	For a vertical strip $\la/\nu$ with $\mu\subset\nu$,
	by removing the $p$-th column of $\nu$ (let $\ti\nu$ denote the resulting shape)
	we get a vertical strip $\ti\la/\ti\nu$
	that satisfies $\mu\subset\ti\nu$
	and $\ti\nu'_{p-1}\ge \la'_p$.
	Conversely, for any vertical strip $\ti\la/\ka$ with $\mu\subset\ka$
	and $\ka'_{p-1}\ge\la'_p$
	and any integer $0\le i\le\la'_p$,
	by adding $i$ boxes in the $p$-th column of $\ka$
	we get the shape $\ka+(1^i)$, for which
	$\la/(\ka+(1^i))$ is a vertical strip.
	Therefore we have a bijection
	\begin{equation}\label{eq:vs_set}
	\{\nu\mid\mu\subset\nu\subset\la,\ \la/\nu\text{: vertical strip}\}
	\simeq
	\{\ka\mid\mu\subset\ka\subset\ti\la,
	\ \ti\la/\ka\text{: vertical strip},
	\ \ka'_{p-1}\ge\la'_p \}
	\times \{0,1,\dots,\la'_p\}
	\end{equation}
	in which $\nu$ corresponds to $(\ti\nu, \nu'_p)$,
	where $\ti\nu$ is $\nu$ with its $p$-th column removed.
	For $\nu$ in the left-hand side of \eqref{eq:vs_set},
	it is easy to see that
	\begin{align*}
	c(\nu/\mu) 	&= c(\ti\nu/\mu) + \DE{\nu'_p>0}, \\
	|\la/\nu| 	&= |\ti\la/\ti\nu| + \la'_p-\nu'_p, \\
	c(\la/\nu)	&= c(\ti\la/\ti\nu) + \DE{\nu'_p < \la'_p},
	\end{align*}
	where we use the notation 
	$\DE{P}=1$ if $P$ is true and $\DE{P}=0$ if $P$ is false
	for a condition $P$.

	Write simply 
	$A = 	
	\{\ka\mid\mu\subset\ka\subset\ti\la,
	\ \ti\la/\ka\text{: vertical strip},
	\ \ka'_{p-1}\ge\la'_p \}$.
	Substituting these to \eqref{eq:itjt},
	we have
	\begin{align*}
	\text{(RHS of \eqref{eq:itjt})}
	&= 	
	\sum_{\ka\in A}
	\sum_{i=0}^{\la'_p}
	t^{c(\ka/\mu) + \DE{i>0}}
	(-1)^{|\ti\la/\ka| + \la'_p-i} 
	t^{c(\ti\la/\ka) + \DE{i < \la'_p}}
	(t-1)^{|\ti\la/\ka| + \la'_p-i-(c(\ti\la/\ka) + \DE{i < \la'_p})} \\
	&= 	\sum_{\ka\in A}
	(-1)^{|\ti\la/\ka|} 
	t^{c(\ka/\mu) + c(\ti\la/\ka)} 
	(t-1)^{|\ti\la/\ka| - c(\ti\la/\ka)}
	\underbrace{
		\sum_{i=0}^{\la'_p}
		(-1)^{\la'_p-i} 
		t^{\DE{i > 0} + \DE{i < \la'_p}} 
		(t-1)^{\la'_p-i-\DE{i < \la'_p}}
	}_{(X)}.
	\end{align*}
	We shall show $(X)=0$.
	Letting $q=\la'_p$ ($>0$) and $j=q-i$ we rewrite $(X)$ as
	\begin{equation}\label{eq:X}
	(X) 
	= 	\sum_{j=0}^{q} 
	(-1)^j
	t^{\DE{j > 0} + \DE{j < q}} 
	(t-1)^{j-\DE{j > 0}}.
	\end{equation}
	It is easy to check $\eqref{eq:X}=0$ when $q=1$.
	When $q\ge 2$, by checking
	\[
	(-1)^q t (t-1)^{q-1}
	+ (-1)^{q-1} t^{2} (t-1)^{q-2}
	= (-1)^{q-1} t (t-1)^{q-2},
	\]
	we can carry induction on $q$ to obtain $\eqref{eq:X}=0$.
	
	Therefore we conclude 
	$\eqref{eq:itjt}=0$ if $\la/\mu\neq\emptyset$,
	finishing the proof of 
	Proposition \ref{theo:(Et,-)_detail}.
\end{proof}

\section{Example}\label{sect:eg}

We display Young diagrams in the French notation.

\begin{exam}\label{exam:I_skew}
	Let $\la/\mu=(3,2,1)/(1)=\sk{3,2,1}{1}$.
	We shall verify \eqref{eq:I_skew} for this $\la/\mu$ by expanding each term into a linear combination of $\{g_\nu\}$.
	We can check
	\begin{equation}\label{eq:321/1}
	\dgs{3,2,1}{1} = \dg{3,2} + \dg{3,1,1} + \dg{2,2,1} - \dg{3,1} - \dg{2,2} - \dg{2,1,1} + \dg{2,1},
	\end{equation}
	using recursively the Pieri formula for skew dual stable Grothendieck polynomials \cite[Theorem 7.1]{1711.09544}
	\begin{equation}\label{eq:skewPieri}
	h_k g_{\mu/\nu} =
	\sum_{\substack{\la/\mu\text{: horizontal strip} \\ \nu/\eta\text{: vertical strip}}}
	(-1)^{k-|\la/\mu|}
	\binom{a(\la\para\mu) - a(\nu'\para\eta') - |\nu/\eta|}{k-|\la/\mu|-|\nu/\eta|}
	g_{\la/\eta},
	\end{equation}
	where 
	$a(\alpha\para\beta)$ is the number of $i\ge 1$ satisfying $\beta_i > \alpha_{i+1}$ and $\beta_{i}>\beta_{i+1}$, and
	the binomial coefficient $\binom{m}{n}$ is considered as $0$ when $n<0$.
	(Note: another way to check \eqref{eq:321/1} is to use \eqref{eq:Iinv:g}.)
	
	Applying $I$ to \eqref{eq:321/1} and using $I(g_\ka)=\sum_{\alpha\subset\ka}g_\alpha$, 
	we compute the first term of \eqref{eq:I_skew} as
	\begin{align}
	I(\dgs{3,2,1}{1})
	&= I(\dg{3,2} + \dg{3,1,1} + \dg{2,2,1} - \dg{3,1} - \dg{2,2} - \dg{2,1,1} + \dg{2,1}) \notag \\
	&= 
	\sum_{\ka\subset\yd{3,2}} g_{\ka} 
	+ \sum_{\ka\subset\yd{3,1,1}} g_{\ka} 
	+ \sum_{\ka\subset\yd{2,2,1}} g_{\ka}
	- \sum_{\ka\subset\yd{3,1}} g_{\ka} 
	- \sum_{\ka\subset\yd{2,2}} g_{\ka} 
	- \sum_{\ka\subset\yd{2,1,1}} g_{\ka}
	+ \sum_{\ka\subset\yd{2,1}} g_{\ka} \notag \\
	&=
	\sum_{\ka\in [\emptyset, \yd{3,2}]\cup[\emptyset, \yd{3,1,1}]\cup[\emptyset, \yd{2,2,1}]} g_\ka.
	\label{eq:I321/1}
	\end{align}
	
	Next we compute the second term of \eqref{eq:I_skew},
	$\sum_{(1)\subset\nu\subset(3,2,1)}g_{\nu/(1)}$.
	Again using \eqref{eq:skewPieri} we have
	\begin{gather*}
	\dgs{3,2,1}{1} = \dg{3,2} + \dg{3,1,1} + \dg{2,2,1} - \dg{3,1} - \dg{2,2} - \dg{2,1,1} + \dg{2,1}, \\
	\dgs{3,2}{1} = \dg{3,1} + \dg{2,2} - \dg{2,1}, \qquad
	\dgs{3,1,1}{1} = \dg{3,1} + \dg{2,1,1} - \dg{2,1}, \qquad
	\dgs{2,2,1}{1} = \dg{2,2} + \dg{2,1,1} - \dg{2,1}, \\
	\dgs{3,1}{1} = \dg{3} + \dg{2,1} - \dg{2}, \qquad
	\dgs{2,2}{1} = \dg{2,1}, \qquad
	\dgs{2,1,1}{1} = \dg{2,1} + \dg{1,1,1} - \dg{1,1}, \\
	\dgs{3}{1} = \dg{2}, \qquad
	\dgs{2,1}{1} = \dg{2} + \dg{1,1} - \dg{1}, \qquad
	\dgs{1,1,1}{1} = \dg{1,1}, \\
	\dgs{2}{1} = \dg{1}, \qquad
	\dgs{1,1}{1} = \dg{1}, \qquad
	\dgs{1}{1} = \dg{}.
	\end{gather*}
	Summing them up, we have
	\begin{equation}\label{eq:sum_nu/1}
	\sum_{\yd{1}\subset\nu\subset\yd{3,2,1}}g_{\nu/\yd{1}} = 
	\sum_{\ka\in [\yd{}, \yd{3,2}]\cup[\yd{}, \yd{3,1,1}]\cup[\yd{}, \yd{2,2,1}]} g_\ka.
	\end{equation}

	Finally we compute the last term of \eqref{eq:I_skew},
	$\sum_{(1)\subset\nu\subset(3,2,1)}g_{(3,2,1)/\nu}$.
	We can check
	\begin{gather*}
	\dgs{3,2,1}{3,2,1} = \dg{}, \qquad
	\dgs{3,2,1}{3,2} = \dg{1}, \qquad
	\dgs{3,2,1}{3,1,1} = \dg{1}, \qquad
	\dgs{3,2,1}{2,2,1} = \dg{1}, \\
	\dgs{3,2,1}{3,1} = \dg{2} + \dg{1,1} - \dg{1}, \qquad
	\dgs{3,2,1}{2,2} = \dg{2} + \dg{1,1} - \dg{1}, \qquad
	\dgs{3,2,1}{2,1,1} = \dg{2} + \dg{1,1} - \dg{1}, \\
	\dgs{3,2,1}{3} = \dg{2,1}, \qquad
	\dgs{3,2,1}{2,1} = \dg{3} + 2\dg{2,1} + \dg{1,1,1} - 2\dg{2} - 2\dg{1,1} + \dg{1}, \qquad
	\dgs{3,2,1}{1,1,1} = \dg{2,1}, \\
	\dgs{3,2,1}{2} = \dg{3,1} + \dg{2,2} + \dg{2,1,1} - 2\dg{2,1}, \qquad
	\dgs{3,2,1}{1,1} = \dg{3,1} + \dg{2,2} + \dg{2,1,1} - 2\dg{2,1}, \\
	\dgs{3,2,1}{1} = \dg{3,2} + \dg{3,1,1} + \dg{2,2,1} - \dg{3,1} - \dg{2,2} - \dg{2,1,1} + \dg{2,1}.
	\end{gather*}
	Summing them up, we have
	\begin{equation}\label{eq:sum_321/nu}
	\sum_{\yd{1}\subset\nu\subset\yd{3,2,1}}g_{\yd{3,2,1}/\nu} = 
	\sum_{\ka\in [\yd{}, \yd{3,2}]\cup[\yd{}, \yd{3,1,1}]\cup[\yd{}, \yd{2,2,1}]} g_\ka.
	\end{equation}
	
	Hence we see $\eqref{eq:I321/1}=\eqref{eq:sum_nu/1}=\eqref{eq:sum_321/nu}$,
	verifying \eqref{eq:I_skew}.
\end{exam}

\begin{rema}
	From \eqref{eq:I321/1} in the example above and
	a Pieri-type formula given in \cite{dualstablesum}
	($I(g_\la)I(g_{(k)}) = \sum_{\mu} g_\mu$,
	summed over $\mu$ satisfying the set difference $\mu\sm\la$ is a horizontal strip of size $\le k$),
	one may expect positivity in the expansions
	$I(g_{\la/\mu}) = \sum_{\nu} \tilde{c}^{\la}_{\mu\nu} g_\nu$ and
	$I(g_\mu)I(g_\nu)=\sum_{\la}\tilde{d}^{\la}_{\mu\nu} g_\la$
	(note that
	$\tilde{c}^{\la}_{\mu\nu} 
	= \sum_{\mu\subset\ka\subset\la} c^{\la}_{\ka,\nu}
	= \sum_{\mu\subset\ka\subset\la} c^{\ka}_{\mu,\nu}$
	and
	$\tilde{d}^{\la}_{\mu\nu} = \sum_{\substack{\alpha\subset\mu \\ \beta\subset\nu}} d^{\la}_{\alpha,\beta}$).
	However, neither hold in general;
	a counterexample for the former is
	$\tilde{c}^{(53221)}_{(321),(321)}=-1$,
	and one for the latter is
	$\tilde{d}^{(5321)}_{(321),(321)}=-1$.
\end{rema}

\bibliographystyle{abbrv}

\end{document}